\documentclass[12pt,a4paper]{article}
\usepackage[english]{babel}
\usepackage[latin1]{inputenc}
\usepackage{times}
\usepackage[T1]{fontenc}

\usepackage{amsmath,amsfonts,amssymb,amsthm} 
\usepackage{tikz}
\usepackage{hyperref}
\usepackage[all]{xy}

\newtheorem{theorem}{Theorem}
\newtheorem{proposition}{Proposition}

\newtheorem{claim}{Claim}

\theoremstyle{definition}
\newtheorem{question}{Question}
\newtheorem{definition}{Definition}

\newcommand{\supp}{\operatorname{supp}}
\newcommand{\Kat}{Kat\v{e}tov}
\newcommand{\Fra}{Fra\"\i ss\'e}
\newcommand{\Flim}{\operatorname{Flim}}
\newcommand{\Aut}{\operatorname{Aut}}
\newcommand{\Emb}{\operatorname{Emb}}

\author{Jan Greb\'ik\footnote{Research supported by GA\v CR project 16-34860L and RVO:
67985840.}\\
{\small Institute of Mathematics, Czech Academy of Sciences}
}
\title{An example of a Fra\"iss\'e class without a Kat\v etov functor}

\begin{document}

\maketitle

\begin{abstract}
We disprove a conjecture from~\cite{KubMas} by showing the existence of a \Fra\ class $\mathcal{C}$ which does not admit a \Kat\ functor.
On the other hand, we show that the automorphism group of the \Fra\ limit of $\mathcal{C}$ is universal, as it happens in the presence of a \Kat\ functor.

\medskip


\end{abstract}

\section{Introduction}
Let us recall some basic facts. A \Fra\ class $\mathcal{C}$ is a countable class of finitely generated structures for some countable language $\mathcal{L}$ together with embeddings which satisfy joint embedding property, amalgamation property, hereditary property. Let us denote by $\sigma\mathcal{C}$ the class of all countably generated structures that are colimits of countable chains from $\mathcal{C}$. Obviously $\mathcal{C}\subseteq \sigma\mathcal{C}$. 
\Fra\ theorem says that there exists a unique homogeneous structure $\Flim(\mathcal{C}) \in \sigma\mathcal{C}$, called the \emph{\Fra\ limit} of $\mathcal{C}$, which is universal for $\sigma\mathcal{C}$. This means that every isomorphism between finitely generated substructures of $\Flim(\mathcal{C})$ extends to an automorphism of $\Flim(\mathcal{C})$, and every structure from $\mathcal{C}$ embeds into $\Flim(\mathcal{C})$.
For more information see, e.g.~\cite{Hod}.

In~\cite{KubMas} the authors define a notion of a \Kat\ functor. Existence of such a functor leads to a uniform way of obtaining the \Fra\ limit and to prove, for example, that $\Aut(\Flim(\mathcal{C}))$ is universal for all $\Aut(X)$ where $X\in \sigma\mathcal{C}$. Let us recall the definition.
\begin{definition}
We say that $K\colon \mathcal{C}\to\sigma\mathcal{C}$ is a \Kat\ functor if 
\begin{itemize}
\item $K$ is a covariant functor,
\item there is a natural transformation $\{\eta_x\}_{x\in\mathcal{C}}$ between $id_{\mathcal{C}}$ and $K$,
$$\xymatrix{
x \ar[d]_f \ar[rr]^{\eta_x} & & K(x) \ar[d]^{K(f)} \\
y \ar[rr]_{\eta_y} & & K(y)}$$
\item every one point extension $y\in \mathcal{C}$ of $x\in \mathcal{C}$ is realized in $K(x)$ over $\eta_x(x)$.
$$\xymatrix{
x \ar[d] \ar[rr] & & K(x) \\
y \ar[rru] & &
}$$
\end{itemize}
\end{definition}
Every \Kat\ functor $K$ naturally extends to a functor $K\colon\sigma\mathcal{C}\to\sigma\mathcal{C}$ with similar properties, i.e., there is a natural transforamtion between $id_{\sigma\mathcal{C}}$ and $K$, every one point extension $y\in \mathcal{C}$ of $x\in \mathcal{C}$ where $x\subseteq X\in\sigma\mathcal{C}$ is realized in $K(X)$ over $\eta_X(x)$. Moreover, iterating the \Kat\ functor $\omega-$many times leads to another \Kat\ functor $K^{\omega}\colon\sigma\mathcal{C}\to\sigma\mathcal{C}$ such that for all $X\in \sigma\mathcal{C}$ it holds that $K(X)=\Flim(\mathcal{C})$. Easy observation (using functoriality) then gives us that $\Aut(X)$ embeds into $\Aut(\Flim(\mathcal{C}))$ and $\Emb(X,Y)$ embeds into $\Emb(\Flim(\mathcal{C}))$ i.e. $K^{\omega}\colon\sigma\mathcal{C}\to \{\Flim(\mathcal{C})\}$ is a faithful functor.  See~\cite{KubMas} for more information. It was an open problem whether there is a \Fra\ class without a \Kat\ functor. In this note we prove that the class of linearly ordered finite sets with colorings of pairs by countably many colors without monochromatic triangles is a \Fra\ class without a \Kat\ functor, however it admits a faithful functor as above.

\section{The construction}
Let us fix some \Fra\ class $\mathcal{C}$ and denote its closure on colimits of countable chains by $\sigma\mathcal{C}$. When we have a \Kat\ functor $K\colon\sigma\mathcal{C}\to\sigma\mathcal{C}$ we will always assume that $K(X)$ is an extension of $X$ i.e. $\eta_X$ is inclusion. Let us denote a one point extension $y\in \mathcal{C}$ of $x\in \mathcal{C}$ by $y = \langle x,t\rangle$ (meaning that $t$ is a single generator). Similarly, $\langle X,t\rangle$ will denote a one point extension of $X\in\sigma\mathcal{C}$, namely, a structure in $\sigma\mathcal{C}$ generated by $X \cup \{t\}$.

\begin{definition}
Let $x,y\in\mathcal{C}$ such that $x\subseteq y$ and $X,C\in \sigma\mathcal{C}$. Consider the following commutative diagram.
$$\xymatrix{
y \ar[rr]^{e_2} & & C \\
x \ar[u]^{i} \ar[rr]_{e_1} & & X \ar[u]_{f}
}$$
Then we say that $C$ is a \emph{homogeneous extension} of $X$ over $x\subseteq y$ if for all $\alpha\in \Aut(X)$ such that $\alpha\upharpoonright x=id_x$, there is $\beta \in \Aut(C)$ such that $\beta \upharpoonright y=id_y$ and the following diagram commutes.
$$\xymatrix{
y \ar[rr]_{e_2} & & C \ar[rr]_{\beta} & & C\\
x \ar[u]^{i} \ar[rr]_{e_1} & & X \ar[u]_{f} \ar[rr]_{\alpha} & & X \ar[u]^{f}\\
}$$
\end{definition}
The most important case is when $C$ is generated by $X\cup y$ and $y$ is generated by one element over $x$. In that case we in fact deal with extending one point extension $\langle x,t\rangle$ to a homogeneous one point extension $\langle X,t\rangle$. 

\begin{proposition}
Assume that there is a \Kat\ functor $K$ for $\mathcal{C}$. Then for all $x\in \mathcal{C}$, all one point extensions $\langle x,t\rangle\in \mathcal{C}$ and all embeddings $e:x\to \Flim(\mathcal{C})$ there exists a homogeneous one point extension $\langle \Flim(\mathcal{C}),t\rangle$ of $\Flim(\mathcal{C})$ over $x\subseteq \langle x,t\rangle$.
\end{proposition}
\begin{proof}
Assume we have the following commutative diagram, where $\alpha$ is an automorphism.
$$\xymatrix{
x \ar[rr]^{e} & & \Flim(\mathcal{C}) \\
x \ar[u]^{id} \ar[rr]_{e} & & \Flim(\mathcal{C}) \ar[u]_{\alpha}
}$$
This diagram can be moved by our \Kat\ functor $K$ to the following commutative diagram
$$\xymatrix{
K(x) \ar[rr]^{K(e)} & & K(\Flim(\mathcal{C})) \\
K(x) \ar[u]^{id} \ar[rr]_{K(e)} & & K(\Flim(\mathcal{C})) \ar[u]_{K(\alpha)}
}$$
Where $K(\alpha)$ is again an automorphism, and these two diagrams are connected by the natural transformation $\eta$. Then $\langle x,t\rangle$ can be embedded to $K(x)$ by the definition of $K$, so we may assume that this embedding is inclusion and put $e(t):=K(e)(t)$. By commutativity we see that $K(\alpha)(e(t))=e(t)$. So $K(\alpha)$ and $K(\alpha)^{-1}$ are invariant on $\langle \Flim(\mathcal{C}),e(t)\rangle$ which is exactly what we needed to prove.
\end{proof}

Another consequence of the existence of a \Kat\ functor is a nontrivial pair of embeddings $e\colon\Flim(\mathcal{C})\to \Flim(\mathcal{C})$  and $E\colon\Aut(\Flim(\mathcal{C}))\to \Aut(\Flim(\mathcal{C}))$ such that the following diagram 
$$\xymatrix{
\Flim(\mathcal{C}) \ar[rr]^{e} & & \Flim(\mathcal{C}) \\
\Flim(\mathcal{C}) \ar[u]^{\alpha} \ar[rr]_{e} & & \Flim(\mathcal{C}) \ar[u]_{E(\alpha)}
}$$
commutes for all $\alpha\in\Aut(\Flim(\mathcal{C}))$. Assume that we have such a pair $(e,E)$, fix $x\subseteq \Flim(\mathcal{C})$ and $t\in \Flim(\mathcal{C})\setminus e[\Flim(\mathcal{C})]$. Define 
$$\mathcal{G}:=\{\alpha\in \Aut(\Flim(\mathcal{C})):E(\alpha)\upharpoonright \langle x,t\rangle=id_{\langle x,t\rangle}\}.$$
This is an open subgroup of $\Aut(\Flim(\mathcal{C}))$ because it is an analytic subgroup (i.e. it has the Baire property) and it has countable index. So there is $x\subseteq x'\subseteq \Flim(\mathcal{C})$ such that $\mathcal{H}:=\{\alpha\colon\alpha\upharpoonright x'=id_{x'}\}\le \mathcal{G}$ i.e. $\langle\Flim(\mathcal{C}),t\rangle$ is a one point homogeneous extension of $\Flim(\mathcal{C})$ over $\langle x', t\rangle$. Hence we may say that for a nontrivial pair of embeddings $(e,E)$ and every one point extension $\langle x,t\rangle$ such that $x\subseteq \Flim(\mathcal{C})$ and $t$ is realized in $\Flim(\mathcal{C})$ over $e[x]$ outside $e[\Flim(\mathcal{C})]$ there are $x'\supseteq x$ and a homogeneous one point extension $\langle \Flim(\mathcal{C}),t \rangle$ of $\Flim(\mathcal{C})$ over $\langle x',t\rangle$. In particular, if $(e,E)$ is nontrivial than there are nontrivial homogeneous one point extensions. 

As it was mentioned in the introduction, iterating $\omega$ many times a fixed \Kat\ functor $K$ leads to the functor $K^{\omega}$ such that $K^{\omega}(X)=\Flim(\mathcal{C})$ for every $X\in \sigma \mathcal{C}$. In particular, $K^{\omega}\colon\sigma\mathcal{C}\to\{\Flim(\mathcal{C})\}$ is a faithful functor.

\begin{theorem}\label{theorem}
There exists a \Fra \ class $\mathcal{C}$ without a \Kat\ functor, yet with a faithful functor from $\sigma\mathcal{C}$ to $\{\Flim(\mathcal{C})\}$.
\end{theorem}

\begin{proof}
Let $Q$ be a countable set. Let us define the class $\mathcal{C}$. An element of $\mathcal{C}$ is a finite set $x$ with a linear order and with a function $c_x:[x]^2\to Q$ such that there are no monochromatic triangles. We will denote the coloring function $c_x$ always by $c$, omitting the subscript $x$. It can be easily seen that $\mathcal{C}$ is a \Fra \ class. 

Let us first prove that there are no homogeneous one point extensions of $\Flim(\mathcal{C})$. Let us fix $x\subseteq \Flim(\mathcal{C})$ and any one point extension $\langle x,t\rangle\in \mathcal{C}$. Let us assume that there is a homogeneous one point extension $\langle \Flim(\mathcal{C}), t\rangle$ of $\Flim(\mathcal{C})$ over $\langle x,t\rangle$. Let us pick any $t_1\in \Flim(\mathcal{C})$ which realizes the same type over $x$ as $t$. Let $q = c((t,t_1)) $. Using the saturation of $\Flim(\mathcal{C})$ we find an element $t_2\in \Flim(\mathcal{C})$ such that $c((t_2,t_1))=q$ and the mapping $\alpha\colon\langle x,t_1\rangle \to \langle x,t_2\rangle$, defined by conditions $\alpha\upharpoonright x=id_x$ and $\alpha(t_1)=t_2$, is an isomorphism. Take any automorphism $\alpha_0\colon\Flim(\mathcal{C}) \to \Flim(\mathcal{C})$ which extends $\alpha$. From the definition of a homogeneous extension, $\alpha_0$ extends to $\langle\Flim(\mathcal{C}),t\rangle$ such that $\alpha_0(t)=t$. Then we have that $\{t,t_1,t_2\}$ is a monochromatic triangle, which is a contradiction. 

From the arguments above it follows that not only there is no \Kat\ functor for $\mathcal{C}$ but there is no nontrivial pair of embeddings $(e,E)$ for $\Flim(\mathcal{C})$ as well.

In order to prove that $\Emb(\Flim(\mathcal{C}))$ is universal for $\sigma\mathcal{C}$, let us define a suitable sequence of \Fra\ classes $\{\mathcal{C}_i\}_{i\le\omega}$. First let us fix a sequence of countable sets $\{Q_i\}_{i\le\omega}$ such that $|Q_{i+1}\setminus Q_i|=\omega$ and $Q_{\omega}=\bigcup_{i<\omega} Q_i$. The definition of $\mathcal{C}_i$ is similar as of $\mathcal{C}$, the only difference is that the set of colors is $Q_i$. We have that $\sigma\mathcal{C}_i\subseteq \sigma\mathcal{C}_{i+1}\subseteq \sigma \mathcal{C}_{\omega}$.
Now for each $i<\omega$ we will find a functor $K_i$ such that 
\begin{itemize}
\item $K_i\colon\mathcal{C}_i\to \sigma\mathcal{C}_{i+1}$,
\item there is a natural transformation $\nu$ for inclusion and $K_i$, i.e., $\nu_x:x\to K_i(x)$ for $x\in \mathcal{C}_i$,
\item for every $x\in \mathcal{C}_i$ every $\mathcal{C}_i-$type over $x$ is realized in $\nu_x[x]$.
\end{itemize}
The functor $K_i$ extends with all its properties (as in the case of \Kat\ functors) to $K_i\colon\sigma\mathcal{C}_i\to \sigma\mathcal{C}_{i+1}$. Once we have this, let us put 
$$K_{\omega}=...\circ K_i\circ ... \circ K_1\circ K_0\colon\sigma\mathcal{C}_0\to\sigma\mathcal{C}_{\omega}.$$
The functor $K_{\omega}$ is correctly defined thanks to the natural transformations for the functors $K_i$, and there is a natural transformation from the inclusion $\sigma\mathcal{C}_0 \subseteq \sigma\mathcal{C}_\omega$ to $K_{\omega}$.
Moreover, we have that $K_{\omega}(X)=\Flim(\mathcal{C}_{\omega})\simeq \Flim(\mathcal{C}_0)$, which proves our claim.

To finish the proof it is enough to describe $K_0$. Let us put $K:=K_0$, $\mathcal{C}:=\mathcal{C}_0$, $Q:=Q_0$, $\mathcal{D}:=\mathcal{C}_1$, $P:=Q_1$ and fix any $p\in P\setminus Q$. Let us define $K$ on objects. Take any $x\in \mathcal{C}$. Assume that we have fixed a linear ordering $<_{Q}$ on $Q$ isomorphic to the natural numbers. Denote by  $\mathcal{O}_x$ the set of all partial proper one-point extensions of $x$ from $\mathcal{C}$. We will use the same letter for a type and for its realization. To every $\xi\in \mathcal{O}_x$ let $\supp(\xi)\subseteq x$ denote the support of $\xi$, that is, the substructure of $x$ to which $\xi$ is added.
Let $K(x) = x \cup \mathcal{O}_x$.
Given an embedding $e \colon x \to y$, define $K(e)$ in the obvious way, namely, a partial type $\xi$ is mapped to the corresponding partial type with support $e[\supp(\xi)]$.
Next we turn $K(x)$ to be an element of $\sigma\mathcal{D}$ in such a way that $K(e)$ remains an embedding of structures whenever $e$ is an embedding.

Extend the coloring by putting $c(v,\xi)=p$ for $v\in x\setminus \supp(\xi)$ and extend the ordering on $\xi$ by declaring $\xi<v$ for $v\in x\setminus \supp(\xi)$ whenever it is consistent with the ordering on $\supp(\xi)\cup \{\xi\}$.
In the next step we define a linear ordering on $K(x)$.
Given $\xi\not=\psi\in \mathcal{O}_x$, let us define $\xi<\psi$ if one of the following conditions is satisfied.
\begin{enumerate}
\item[(1)] there is $v\in x$ such that $\xi<v<\psi$,
\item[(2)] condition (1) fails and $|\supp(\xi)|<|\supp(\psi)|$,
\item[(3)] conditions (1), (2) fail and the biggest $v\in \supp(\xi)\triangle \supp(\psi)$ satisfies $v\in \supp(\xi)$,
\item[(4)] none of the above is satisfied and for the biggest $v\in \supp(\xi)=\supp(\psi)$ for which $c_{\xi}((v,\xi))\not= c_{\psi}((v,\psi))$ we have that $c_{\xi}((v,\xi))<_{Q} c_{\psi}((v,\psi))$.
\end{enumerate}
It remains to define a coloring on $K(x)$ extending the coloring of $x$.
In order to do this, let us define an equivalence relation on pairs of elements from $\mathcal{O}_x$. We say that $\{\xi_0<\psi_0\}\sim \{\xi_1<\psi_1\}$ if there is an isomorphism between $\supp(\xi_i)\cup \supp(\psi_i)$ where $i \in \{0,1\}$ whose extension to $\{\xi_i<\psi_i\}$ remains an isomorphism. It is clear from the definition that if $\{\xi_0<\psi_0\}\sim \{\xi_1<\psi_1\}$ then $\supp(\xi_0) \simeq \supp(\xi_1)$ and $\supp(\psi_0) \simeq \sup(\psi_1)$.  Furthermore, the isomorphisms are unique, because of the linear orderings. It follows immediately:

\begin{claim}\label{ClaimOne}
For $f\colon x\to y$ we have that $\{\xi_0<\psi_0\}\sim \{\xi_1<\psi_1\}$ in $\mathcal{O}_x$ iff $\{K(f)(\xi_0)<K(f)(\psi_0)\}\sim \{K(f)(\xi_1)<K(f)(\psi_1)\}$ in $\mathcal{O}_y$.
\end{claim}

Let us now color the equivalence classes by induction on the size of $x$.
Assume that for all sets of size $<n$ we have already defined the coloring and take $x$ such that $|x|=n$. For the equivalence class of a pair $\{\xi_0<\psi_0\}$ use the color $r\in P$ if there is an embedding $f:y\to x$ such that $\{\xi_0<\psi_0\}$ is in the image of $K(f)$ and their preimage is colored by $r$.
This is well-defined by Claim~\ref{ClaimOne}.
Color the remaining equivalence classes by different (not already used) colors so that  infinitely many colors in $P$ are still left.

We see that whenever $f$ is an embedding of $\mathcal{C}$-structures then $K(f)$ respects both the ordering and the coloring. To finish the proof we must show that there are no monochromatic triangles in $K(x)$. Assume that there is one $\xi<\psi<\mu$.
This means that all pairs are in the same equivalence class.
There are isomorphisms $i,j$ which witness that for $\{\xi<\psi\}\sim \{\xi<\mu\}$ and $\{\xi<\mu\}\sim \{\psi<\mu\}$. We know that then $i$ is identity on $\supp(\xi)$ and $j$ is identity on $\supp(\mu)$.
Because the triangle is not degenerated, there is $v\in \supp(\psi)\setminus \supp(\mu)$, for such $v$ it holds that $w=i(v)\not= v$ and $j(i(v))=j(w)=w\not= v$. There must be some $z\in \supp(\xi)$ such that $j(z)=v$. Then we have that $c(\{v,z\})=c(\{w,z\})$ which is witnessed by $i$, and $c(\{v,w\})=c(\{z,w\})$ which is witnessed by $j$. This is a contradiction, because then $\{v,w,z\}$ is a monochromatic triangle in $x\in \mathcal{C}$.
\end{proof}

\section{Final remarks}

Our result suggests the following question: Is there a \Fra\ class without a faithful functor like in Theorem \ref{theorem}?

It can be easily verified that if $\sigma\mathcal{C}$ has push-outs, then for all $x\in\mathcal{C}$, all one-point extensions $\langle x,t\rangle\in\mathcal{C}$ and all embeddings $f:x\to \Flim{\mathcal{C}}$ there is a one-point homogenenous extension $\langle \Flim{\mathcal{C}},t\rangle$ of $\Flim{\mathcal{C}}$ over $\langle x,t\rangle$. That follows from the fact that the extension of $\Flim{\mathcal{C}}$ can be defined to be push out of the following diagram 
$$\xymatrix{
\langle x,t\rangle \ar[rr]^{i_0} & & \langle \Flim{\mathcal{C}},t\rangle \\
x \ar[u]^{i} \ar[rr]_{f} & & \Flim{\mathcal{C}} \ar[u]_{f_0}
}$$
one can easily verify that it is really an homogeneous one-point extension. This construction works for arbitrary object $X\in\sigma\mathcal{C}$ instead of $\Flim{\mathcal{C}}$. Let us finally state that the category of locally finite countable groups does not have a push out which brings us to the following concrete open question.

\begin{question}
Is there a \Kat\ functor for the \Fra\ class of finite groups? What can we say about one-point homogeneous extensions in this category?
\end{question}

\paragraph{Acknowledgments.} The author would like to thank Egbert Th\"ummel for careful reading the manuscript and for useful comments. The author is also grateful to Wies\l aw Kubi\'s for bringing the problem, fruitful discussions and help.

\end{document}